\documentclass[11pt]{article}
\usepackage{a4wide}
\usepackage{amsmath}
\usepackage{amsfonts}
\usepackage{amsthm}
\usepackage{amssymb}
\usepackage[all]{xy}
\usepackage{flafter}
\usepackage[utf8]{inputenc}
\usepackage{multirow}
\usepackage{enumitem}
\usepackage{color}
\newcommand{\N}{{\mathbb N}}
\newcommand{\Z}{{\mathbb Z}}
\newcommand{\Q}{{\mathbb Q}}
\newcommand{\R}{{\mathbb R}}
\newcommand{\C}{{\mathbb C}}
\newcommand{\F}{{\mathbb F}}
\newcommand{\semi}{{\mathbb o}}
\DeclareMathOperator{\lie}{{\mathfrak g}}

\DeclareMathOperator{\Fix}{{\rm Fix}}
\DeclareMathOperator{\Aut}{{\rm Aut}}

\DeclareMathOperator{\Aff}{{\rm Aff}}
\DeclareMathOperator{\Sp}{{\rm sp}}
\DeclareMathOperator{\aff}{{\rm aff}}
\DeclareMathOperator{\Endo}{Endo}
\DeclareMathOperator{\Dim}{{\rm dim}}
\DeclareMathOperator{\HPer}{{\rm HPer}}
\DeclareMathOperator{\Per}{{\rm Per}}
\DeclareMathOperator{\Tr}{{\rm Tr}}

\DeclareMathOperator{\dd}{\delta}
\DeclareMathOperator{\D}{{\mathfrak D}}
\DeclareMathOperator{\A}{{\mathfrak A}}
\DeclareMathOperator{\B}{{\mathfrak B}}

\newtheorem{theorem}{Theorem}[section]
\newtheorem{proposition}[theorem]{Proposition}
\newtheorem{lemma}[theorem]{Lemma}
\newtheorem{corollary}[theorem]{Corollary}
\newtheorem{definition}[theorem]{Definition}
\newtheorem{remark}[theorem]{Remark}

\newcommand{\ds}{\displaystyle}
\newcommand{\ra}{\rightarrow}

\makeatletter
\def\blfootnote{\xdef\@thefnmark{}\@footnotetext}
\makeatother
\begin{document}
\title{\bf A note on homotopy minimal periods for hyperbolic maps on infra-nilmanifolds}
\author{Karel Dekimpe and Gert-Jan Dugardein\thanks{Research supported by the research fund of the KU Leuven}\\
KU Leuven Kulak, E.\ Sabbelaan 53, B-8500 Kortrijk}
\date{\today}
\maketitle

\begin{abstract}
In this paper, we show that for every non-nilpotent hyperbolic map $f$ on an infra-nilmanifold, the set $\HPer(f)$ is cofinite in $\N$.  This is a generalization of a similar result for expanding maps in \cite{lz07-1}. Moreover, we prove that for every nilpotent map $f$ on an infra-nilmanifold, $\HPer(f)=\{1\}$.
\end{abstract}
\section{Infra-nilmanifolds}
Let $f:X\rightarrow X$ be a map on a topological space $X$. We say that $x\in X$ is a periodic point of $f$ if $f^n(x)=x$ for some 
postive integer $n$. If this is the case, we say that this positive integer $n$ is the pure period of $x$ if $f^l(x)\neq x$ for all $l<n$. In this paper we
will study these periodic points when $X$ is an infra-nilmanifold and we will show that for a large class of maps $f$ on
such manifolds, there exists a postive integer $m$ such that any map $g$ homotopic to $f$ admits points of pure period $k$ for any 
$k\in [m,+\infty)$.  In this first section we will recall the necessary background on the class infra-nilmanifolds and their maps. 
In the next section, we give a more detailed description of the theory of fixed and periodic points. The third and last section is devoted to the 
proof of our main result.

\medskip

Every infra-nilmanifold is modeled on a connected and simply connected nilpotent Lie group. Given such a Lie group $G$, we consider its group of affine transformations
$\Aff(G)\!=\! G\semi \Aut(G)$, which admits a natural left action on the Lie group $G$:
\[ \forall (g,\alpha)\in \Aff(G),\, \forall h \in G: \;\;^{(g,\alpha)}h= g \alpha(h).\]

Note that when $G$ is abelian, $G$ is isomorphic to $\R^n$ for some $n$ and and $\Aff(G)$ is the usual affine group $\Aff(\R^n)$ with its usual action on the affine space $\R^n$. 
Let $p:\Aff(G)=G\semi \Aut(G) \to \Aut(G)$ denote the natural projection onto the second factor. 

\begin{definition} A subgroup $\Gamma \subseteq \Aff(G)$ is called almost-crystallographic if and only if $p(\Gamma)$ is finite and $\Gamma\cap G$ is 
a uniform and discrete subgroup of $G$. The finite group $F=p(\Gamma)$ is called the holonomy group of $\Gamma$.
\end{definition}

The action of $\Gamma$ on $G$ is properly discontinuous and cocompact and when $\Gamma$ is torsion-free, this action becomes a free action, from which 
we can derive that the resulting quotient space $\Gamma\backslash G$ is a compact manifold with fundamental group $\Gamma$. 

\begin{definition}
A torsion-free almost-crystallographic group $\Gamma\subseteq \Aff(G) $ 
is called an almost-Bieberbach group, and the corresponding manifold $\Gamma\backslash G$ is called an infra-nilmanifold (modeled on $G$).  
\end{definition} 

When the holonomy group is trivial, $\Gamma$ will be a lattice in $G$ and the corresponding manifold $\Gamma\backslash G$ is a nilmanifold. 
When $G$ is abelian, $\Gamma$ will be called a Bieberbach group and $\Gamma\backslash G$ a compact flat manifold. 
When $G$ is abelian and the holonomy group of $\Gamma$ is trivial, then $\Gamma$ is just a lattice in some $\R^n$ and $\Gamma\backslash G$ is a torus.

\medskip

Now, define the semigroup $\aff(G)=G\semi \Endo(G)$, where $\Endo(G)$ is the set of continuous endomorphisms of $G$. 
Note that $\aff(G)$ acts on $G$ in a similar way as $\Aff(G)$, i.e.\ any element $(\dd,\D)$ of $\aff(G)$ can be seen as a self-map of $G$:
\[ (\dd,\D): \; G \rightarrow G:\; h \mapsto \dd \D(h)\]
and we will refer to $(\dd,\D)$ as an affine map of $G$. 
One of the nice features of infra-nilmanifolds is that any map on a infra-nilmanifold is homotopic to a map which is induced by an affine 
map of $G$. One can prove this by using the following result by K.B. Lee.

\begin{theorem}[K.B.\ Lee \cite{lee95-2}]
\label{leemaps} Let $G$ be a connected and simply connected nilpotent Lie group and suppose that $\Gamma, \Gamma'\subseteq \Aff(G)$ are two almost-crystallographic groups modeled on $G$. 
Then for any homomorphism $\varphi: \Gamma\rightarrow \Gamma'$ there 
exists an element $  (\dd, \D)\in \aff(G)$ such that 
\[ \forall \gamma \in \Gamma: \; \varphi(\gamma) (\dd, \D) =  (\dd, \D) \gamma.\] 
\end{theorem}

Note that we can consider the equality $ \varphi(\gamma) (\dd, \D) =  (\dd, \D) \gamma$ in $\aff(G)$, since $\Aff(G)$ is contained in $\aff(G)$. With this equality in mind, it is easy to see that the affine map $(\dd,\D)$ induces a well-defined map \[\overline{(\dd,\D)}: \Gamma \backslash G \rightarrow \Gamma' \backslash G: \; \Gamma h \rightarrow \Gamma' \dd \D(h),\]
which exactly induces the morphism $\varphi$ on the level of the fundamental groups.

\medskip

On the other hand, if we choose an arbitrary map $f:\Gamma\backslash G\ra \Gamma'\backslash G$ between two infra-nilmanifolds and choose a lifting $\tilde{f}:G \to G$ of $f$, then there exists a morphism $\tilde{f}_\ast:\Gamma\to \Gamma'$ such that $\tilde{f}_\ast(\gamma) \circ \tilde{f} = \tilde{f}\circ \gamma$, for all $\gamma\in \Gamma$. By Theorem~\ref{leemaps}, an affine map $(\dd,\D)\in \aff(G)$ exists which also satisfies $\tilde{f}_\ast(\gamma) \circ (\dd,\D)= (\dd,\D)\circ \gamma$ for all $\gamma\in \Gamma$. Therefore, the induced map $\overline{(\dd,D)} $ and 
$f$ are homotopic. We will call $(\dd,\D)$ an affine homotopy lift of $f$. 


\medskip

We end this introduction about infra-nilmanifolds with the definition of a hyperbolic map on an infra-nilmanifold. 
We will denote by $\D_\ast$ the Lie algebra endomorphism induced by $\D$ on the Lie algebra $\lie$ associated to $G$.

\begin{definition}
Let $M$ be an infra-nilmanifold and $f:M\to M$ be a continuous map, with $(\dd,\D)$ as an affine homotopy lift. We say that $f$ is a hyperbolic map if $\D_\ast$ has no eigenvalues of modulus $1$.
\end{definition}

\begin{remark}
The map $\D$, and hence also $\D_\ast$ depends on the choice of the lift $\tilde{f}$. Once the lift $\tilde{f}$ is fixed, and hence the 
morphism $\tilde{f}_\ast$ is fixed, the $\D$--part of the map $(\dd,\D)$ in Theorem~\ref{leemaps} is also fixed (although the $\dd$--part is not unique in general).
It follows that $f$ determines $\D$ only up to an inner automorphism of $G$. But as inner automorphisms have no effect on the eigenvalues of $\D_\ast$
(in case of a nilpotent Lie group $G$) the notion of a hyperbolic map is well defined.
\end{remark}

Two important classes of maps on infra-nilmanifolds which are hyperbolic are the expanding maps and the Anosov diffeomorphisms.

\begin{remark}\label{RemarkHyp}
Due to Lemma 4.5. in \cite{ddm05-1}, it is known that every nowhere expanding map on an infra-nilmanifold only has eigenvalues $0$ or eigenvalues of modulus $1$. This means that every hyperbolic map for which $\D_\ast$ is not nilpotent has an eigenvalue of modulus strictly bigger than $1$.
\end{remark}

\section{Nielsen theory, dynamical zeta functions and $\HPer(f)$}
Let $f:X\to X$ be self-map of a compact polyhedron $X$. There are different ways to assign integers to this map $f$ that give information about the fixed points of $f$. One of these integers is the Lefschetz number $L(f)$ which is defined as 
\[ L(f)= \sum_{i=0}^{{\Dim}\;X} (-1)^i {\Tr} \left( f_{\ast,i} : H_i(X,\R) \to H_i(X,\R)\right).\]

In our situation, the space $X=\Gamma \backslash G$ will be a infra--nilmanifold, which is an aspherical space, and hence the (co)homology of the space 
$X=\Gamma\backslash G$ equals the (co)homology of the group $\Gamma$. It follows that in this case we have (see also \cite[page 36]{jm06-1})
\begin{eqnarray*}
  L(f) & = &  \sum_{i=0}^{{\Dim}\;X} (-1)^i {\Tr} \left( f_{\ast,i} : H_i(\Gamma,\R) \to H_i(\Gamma,\R)\right) \\
       & = & \sum_{i=0}^{{\Dim}\;X} (-1)^i {\Tr} \left( f^\ast_i : H^i(\Gamma,\R) \to H^i(\Gamma,\R)\right)
\end{eqnarray*}

The Lefschetz fixed point theorem states that if $L(f)\neq 0$, then $f$ has at least one fixed point. Because the Lefschetz number is only defined in terms of 
(co)homology groups, it remains invariant under a homotopy and hence, if $L(f)\neq 0$, the Lefschetz fixed point theorem guarantees that any map homotopic to $f$ also 
has at least one fixed point.

\medskip

Another integer giving information on the fixed points of $f$ is the Nielsen 
number $N(f)$. It is a homotopy-invariant lower bound for the number of fixed points of $f$. To define $N(f)$, fix a reference lifting $\tilde{f}$ of $f$ with respect to a universal cover $(\tilde{X},p)$ of $X$ and denote the group of covering transformations by $\mathcal{D}$. For $\alpha\in \mathcal{D}$, the sets $p(\Fix(\alpha \circ \tilde{f}))$ form a partition of the fixed point set $\Fix(f)$. These sets are called fixed point classes. By using the fixed point index, we can assign an integer to each fixed point class in such a way that if a non-zero integer is assigned, the fixed point class cannot completely vanish under a homotopy. Such a non-vanishing fixed point class will be called essential and $N(f)$ is defined as the number of essential fixed point classes of $f$.

\medskip

By definition, it is clear that $N(f)$ will indeed be a homotopy-invariant lower bound for the number of fixed points of $f$. Hence, in general, $N(f)$ will give more information about the fixed points of $f$ than $L(f)$. The downside, however, is that Nielsen numbers are often much harder to compute than Lefschetz numbers, because the fixed point index can be a tedious thing to work with. Luckily, on infra--nilmanifolds there exists an algebraic formula to compute $N(f)$, which makes them 
a convenient class of manifolds to study Nielsen theory on. More information on both $L(f)$ and $N(f)$ can be found in e.g. \cite{brow71-1}, \cite{jian83-1} and \cite{kian89-1}.

\medskip

By using the Lefschetz and Nielsen numbers of iterates of $f$ as coefficients, it is possible to define the so-called dynamical zeta functions. The Lefschetz zeta function was introduced by S.~Smale in \cite{smal67-1}:\[ L_f(z)=\exp\left( \sum_{k=1}^{+\infty}\frac{L(f^k)}{k}z^k\right).\]In his paper, Smale also proved that the Lefschetz zeta function is always rational for self-maps on compact polyhedra. 

\medskip

The proof is actually quite straightforward. Let the $\lambda_{ij}$'s denote 
the eigenvalues of  $f_{\ast}^i:H^i(X,\R)\rightarrow H^i(X,\R)$, with $j\in \{1,\dots, \Dim(H^i(X,\R))\}$. Because the trace of a matrix is the sum of the eigenvalues, we find $$L_f(z)=\exp\left( \sum_{k=1}^{+\infty}\left(\sum_{i=0}^{\Dim\;X}(-1)^i\sum_{j=1}^{\Dim\;H^i(X)}\lambda_{ij}^k\right)\frac{z^k}{k}\right).$$By reordering the terms and by using the fact that $$\sum_{k=1}^{+\infty}\frac{a^kz^k}{k}=-\log(1-az), \textrm{ for } |z|<|a|^{-1},$$it is easy to derive that \begin{equation}\label{eqLfz}
L_f(z)=\prod_{i=0}^{\Dim\;X}\prod_{j=1}^{\Dim\;H^i(X)}(1-\lambda_{ij}z)^{(-1)^{i+1}}.
\end{equation}

\begin{remark}\label{voor in rode deel} Suppose that $\Lambda$ is a lattice of a connected and simply connected nilpotent Lie group $G$ and 
$f:\Lambda \backslash G \rightarrow \Lambda\backslash G$ is a self map of the nilmanifold $\Lambda \backslash G$ with affine homotopy lift $(\dd,\D)$. Let 
$\D_\ast$ be the induced linear map on the Lie algebra ${\mathfrak g}$ of $G$ as before.
The main result of \cite{nomi54-1} states that there are natural isomorphisms
\[ H^i (\Lambda , \R) \cong H^i(\Lambda\backslash G ,\R) \cong H^i({\mathfrak g},\R).\]
The naturality of these automorphisms implies that there is a commutative diagram
\[ 
\xymatrix{ H^i (\Lambda , \R)\ar[r]^{\cong}\ar[d]_{f_i^*}  &  H^i({\mathfrak g},\R)\ar[d]^{\D_\ast^i}\\
H^i (\Lambda , \R)\ar[r]_{\cong}  &  H^i({\mathfrak g},\R)}\]
Here $\D_\ast^i$ is the map induced by $\D_\ast$ on the $i$--th cohomology space of ${\mathfrak g}$. Recall, that the cohomology of 
${\mathfrak g}$ is defined as the cohomology of a cochain complex, where the $i$--th term is Hom$(\bigwedge^i{\mathfrak g}, \R)=(\bigwedge^i{\mathfrak g})^\ast$, 
the dual space of $\bigwedge^i{\mathfrak g}$. So, 
$\D_\ast^i$ is induced by the dual map of $\bigwedge^i \D_\ast$. Since this dual map and $\bigwedge^i \D_\ast$ have the same eigenvalues, it follows 
that the set of eigenvalues of $\D_\ast^i$, hence also the set of eigenvalues $\lambda_{i,j}$ of $f_i^*$ in expression \eqref{eqLfz}, 
is a subset of the set of eigenvalues of $\bigwedge^i \D_\ast: \bigwedge^i {\mathfrak g} \rightarrow \bigwedge^i {\mathfrak g}$. 
(This fact is also reflected in the formula obtained in \cite[Theorem 23]{fels00-2})
\end{remark}

The Nielsen zeta function was introduced by A.~Fel'shtyn in \cite{fels88-1,fp85-1} and is defined in a similar way as the Lefschetz zeta function: \[ N_f(z)=\exp\left( \sum_{k=1}^{+\infty}\frac{N(f^k)}{k}z^k\right).\]It is known that this zeta function does not always have to be a rational function. A counterexample for this can be found in \cite{fels00-2}, e.g. in Remark 7.

\medskip

For self-maps on infra-nilmanifolds, however, the Nielsen zeta function will always be rational. To prove this, one can exploit the fact that $N(f)$ and $L(f)$ are very closely related. In \cite{dd13-2}, we defined a subgroup $\Gamma_+$ of $\Gamma$, which equals $\Gamma$ or is of index $2$ in $\Gamma$. The precise definition is not of major significance for the rest of this paper. However, it allowed us to write $N_f(z)$ as a function of $L_f(z)$, if $\Gamma=\Gamma_+$ and as a combination of $L_f(z)$ and $L_{f_+}(z)$, if $[\Gamma:\Gamma_+]=2$. Here, $f_+:\Gamma_+\backslash G\to\Gamma_+\backslash G$ is a lift of $f$ to the $2$-folded covering space $\Gamma_+\backslash G$ of $\Gamma\backslash G$. The following theorem, together with the fact that Lefschetz zeta functions are always rational, therefore proves the rationality of Nielsen zeta functions for infra-nilmanifolds.

\begin{theorem}\cite[Theorem 4.6]{dd13-2}\label{Nielsenzeta}
Let $M=\Gamma\backslash G$ be an infra-nilmanifold and let $f:M\to M$ be a self-map with affine homotopy lift $(\dd,\D)$. Let $p$ denote the number of positive real eigenvalues of $\D_\ast$ which are strictly greater than $1$ and 
let $n$ denote the number of negative real eigenvalues of $\D_\ast$ which are strictly less than $-1$. Then we have the following table of equations:
\renewcommand{\arraystretch}{1.9}
\begin{center}\begin{tabular}{ |c|c|c|c|c| }
\cline{2-5}
\multicolumn{1}{l}{ }
 &  \multicolumn{1}{|c|}{$p$ even, $n$ even}
 & \multicolumn{1}{|c|}{$p$ even, $n$ odd}& \multicolumn{1}{|c|}{$p$ odd, $n$ even}& \multicolumn{1}{|c|}{$p$ odd, $n$ odd} \\
\cline{1-5}
$\Gamma= \Gamma_+$ & $N_f(z)= L_f(z) $& $N_f(z)= \ds \frac{1}{L_{f}(-z)} $ &$N_f(z)= \ds \frac{1}{L_{f}(z)}$ & $N_f(z)= L_f(-z) $\\[1ex]
\cline{1-5}
$\Gamma\neq \Gamma_+$ &$N_f(z)= \ds \frac{L_{f_+}(z)}{L_{f}(z)}$ & $N_f(z)= \ds \frac{L_{f}(-z)}{L_{f_+}(-z)}$ &$N_f(z)= \ds \frac{L_{f}(z)}{L_{f_+}(z)}$ & $N_f(z)= \ds \frac{L_{f_+}(-z)}{L_{f}(-z)}$ \\[1ex]
\cline{1-5}
\end{tabular}
 \end{center}
\end{theorem}

Moreover, this theorem also tells us that we can write $N_f(z)$ in a similar form as in equation (\ref{eqLfz}), since every Lefschetz zeta function is of this form. More information about dynamical zeta functions can be found in \cite{fels00-2}.

\medskip

Closely related to fixed point theory, is periodic point theory. We call $x\in X$ a periodic point of $f$ if there exists a positive integer $n$, such that $f^n(x)=x$. 
Of course, when $f^n(x)=x$, this does not automatically imply that the actual period of $x$ is $n$. 
E.g., it is immediately clear that every fixed point is also a periodic point of period $n$, for all $n>0$. In order to exclude these points, we define the set of 
periodic points of pure period $n$:$$P_n(f)=\{x\in X\; \|\; f^n(x)=x \textrm{ and } f^k(x)\neq x, \forall k|n\}.$$
The set of homotopy minimal periods of $f$ is then defined as the following subset of the positive integers:$$\HPer(f)=\bigcap_{f\simeq g}\{n| P_n(g)\neq \emptyset\}.$$

This set has been studied extensively, e.g. in \cite{ablss95-1} for maps on the torus, in \cite{jm02-1} for maps on nilmanifolds and in \cite{fl14-1}, \cite{lz07-1} for maps on infra-nilmanifolds.

\medskip

Just as Nielsen fixed point theory divides $\Fix(f)$ into different fixed point classes, Nielsen periodic point theory divides $\Fix(f^n)$ into different fixed point classes, for all $n>0$ and looks for relations between fixed point classes on different levels. This idea is covered by the following definition.

\begin{definition}
Let $f:X \to X$ be a self-map. If $\F_k$ is a fixed point class of $f^k$, then $\F_k$ will be contained in a fixed point class $\F_{kn}$ of $(f^k)^n$, for all $n$. We say that $\F_k$ boosts to $\F_{kn}$. On the other hand, we say that $\F_{kn}$ reduces to $\F_k$.
\end{definition}

An important definition that gives some structure to the boosting and reducing relations is the following.

\begin{definition}
A self-map $f:X\to X$ will be called essentially reducible if, for all $n,k$, essential fixed point classes of $f^{kn}$ can only reduce to essential fixed point classes of $f^k$. A space $X$ is called essentially reducible if every self-map $f:X\to X$ is essentially reducible.
\end{definition}

It can be shown that the fixed point classes for maps on infra-nilmanifolds always have this nice structure for their boosting and reducing relations.

\begin{theorem}[\cite{lz07-1}]
Infra-nilmanifolds are essentially reducible.
\end{theorem}

One of the consequences of having this property, is the following.

\begin{theorem}[\cite{ablss95-1}]\label{ThmABLSS}
Suppose that $f$ is essentially reducible and suppose that $$N(f^k)>\sum_{p \textrm{ prime},\, p|k} N(f^{k/p}),$$then $k\in \HPer(f)$.
\end{theorem}

The idea of this theorem is actually quite easy to grasp. Because maps on infra-nilmanifolds are essentially reducible, every reducible essential fixed point class on 
level $k$ will reduce to an essential fixed point class on level $\frac{k}{p}$, with $p$ a prime divisor of $k$. Therefore, the condition 
$$N(f^k)>\sum_{p \textrm{ prime},\, p|k} N(f^{k/p})$$
actually tells us that there is definitely one irreducible essential fixed point class on level $k$, which means that there is at least one periodic point of pure period $k$.

\medskip

For this paper, this is all we need to know about Nielsen periodic point theory. More information about Nielsen periodic point theory in general can be found in \cite{hk97-1}, \cite{jm06-1} or \cite{jian83-1}.
\section{$\HPer(f)$ for hyperbolic maps on infra-nilmanifolds}

\subsection{Non-nilpotent case}

We begin with the following definition, which tells us something about the asymptotic behavior of the sequence $\left\{N(f^k)\right\}_{k=1}^\infty$.
\begin{definition}
The asymptotic Nielsen number of $f$ is defined as$$N^{\infty}(f)=\max\left\{1,\limsup_{k\to \infty}N(f^k)^{\frac{1}{k}}\right\}.$$
\end{definition}

By $\Sp(A)$ we mean the spectral radius of the matrix or the operator $A$. It equals the largest modulus of an eigenvalue of $A$.

\begin{theorem}\cite[Theorem 4.3]{fl13-1}\label{ThmAsymNielsen}
For a continuous map $f$ on an infra-nilmanifold, with affine homotopy lift $(\dd,\D)$, such that $\D_\ast$ has no eigenvalue $1$, we have$$N^\infty(f)=\Sp(\bigwedge \D_\ast).$$ 
\end{theorem}

If $\{\nu_i\}_{i\in I}$ is the set of eigenvalues of $\D_\ast$, we know that \[\Sp(\bigwedge \D_\ast)=\left\{
\begin{array}{ccc}
\prod_{|\nu_i|>1}|\nu_i| & \textrm{ if }& \Sp(\D_\ast)>1 \\
1 & \textrm{ if } & \Sp(\D_\ast)\leq 1\end{array}\right..\]Therefore, we have the following corollary of Theorem \ref{ThmAsymNielsen}. 

\begin{corollary}
Let $f$ be a hyperbolic, continuous map on an infra-nilmanifold. Let $(\dd,\D)$ be an affine homotopy lift of $f$ en let $\{\nu_i\}_{i\in I}$ be the set of eigenvalues of $\D_\ast$. When $\D_\ast$ is not nilpotent, then$$N^\infty(f)=\prod_{|\nu_i|>1}|\nu_i|.$$
\end{corollary}
\begin{proof}
When $\D_\ast$ is not nilpotent, we know by Remark \ref{RemarkHyp} that $\Sp(\D_\ast)>1$. Because $f$ is hyperbolic, $1$ is certainly not an eigenvalue of $\D_\ast$ and therefore, we can use the result of Theorem \ref{ThmAsymNielsen}.
\end{proof}

Because of Theorem \ref{Nielsenzeta}, we know that $N_f(z)$ can be written as the quotient of Lefschetz zeta functions. Since every Lefschetz zeta function on a compact polyhedron is of the form $$L_f(z)=\prod_{i=1}^{m}(1-\mu_i z)^{\gamma_i},$$with $\mu_i \in \C$ and $\gamma_i\in \{1,-1\}$, the same will hold for $N_f(z)$. Also, it is easy to check that $$N_f(z)=\prod_{i=1}^{n}(1-\lambda_i z)^{-\varepsilon_i} \Rightarrow N(f^k)=\sum_{i=1}^n \varepsilon_i \lambda_i^k,$$for all $k\in \N$.

\medskip

In Remark~\ref{voor in rode deel} we already mentioned the fact that for nilmanifolds the $\mu_i$'s appearing in the expression for $L_f(z)$ are 
eigenvalues of $\bigwedge \D_\ast$. We now claim that the same holds for maps on infra--nilmanifolds. 
Consider an infra--nilmanifold $\Gamma\backslash G$ and a self-map $f$ of $\Gamma\backslash G$ with affine homotopy lift $(\dd,\D)$.  Without loss 
of generality, we may assume that $f=\overline{(\dd,\D)}$.  We now fix a fully characteristic subgroup $\Lambda$ of finite index in $\Gamma$ that is contained in $G$ (e.g.\ see \cite{ll06-1}). Hence for the induced morphism $f_\ast:\Gamma \rightarrow \Gamma$ we have that 
$f_\ast(\Lambda) \subseteq \Lambda$. It follows that $(\dd, \D)$ also induces a map $\hat{f}$ on the nilmanifold $\Lambda \backslash G$ and that 
$\hat{f}_\ast= f_{\ast|\Lambda}$.  By Theorem III 10.4 in \cite{brow82-1} we know that the restriction map induces an isomorphism 
${\rm res}:H^{i}(\Gamma,\Q) \rightarrow H^i(\Lambda,\Q)^{\Gamma/\Lambda}$. As the restriction map is natural, we obtain the following commutative diagram:
\[ 
\xymatrix{ H^{i}(\Gamma,\Q) \ar[r]^-{{\rm res}}_-{\cong}\ar[d]_{f_\ast^i}&  H^i(\Lambda,\Q)^{\Gamma/\Lambda}\ar[d]^{\hat{f}^i_\ast}\\
           H^{i}(\Gamma,\Q) \ar[r]^-{{\rm res}}_-{\cong}&  H^i(\Lambda,\Q)^{\Gamma/\Lambda}}
\]
It follows that each of the eigenvalues of $f_\ast^i$ is also an eigenvalue of $\hat{f}^i_\ast$. Since the latter ones are all eigenvalues of 
$\bigwedge^i \D_\ast$, by Remark~\ref{voor in rode deel}, it follows that also all eigenvalues of $f_\ast^i$ are eigenvalues of $\bigwedge^i \D_\ast$. This means that the $\mu_i$'s appearing in the expression for $L_f(z)$ are 
eigenvalues of $\bigwedge \D_\ast$ and of course, because $f_+$ has the same affine homotopy lift as $f$, the same applies for $L_{f_+}(z)$.

\medskip
By Theorem~\ref{Nielsenzeta}, we know that $N_f(z)$ can be written as a combination of $L_f(z)$ and possibly $L_{f_+}(z)$, or as a combination of $L_f(-z)$ and possibly $L_{f_+}(-z)$. In the first case, by the previous discussion, we see that all $\lambda_i$'s in the expression for $N_f(z)$ are eigenvalues of $\bigwedge \D_\ast$. In the latter case, all $\lambda_i$'s are the opposite of eigenvalues of $\bigwedge \D_\ast$. All in all, this means we can write $$N(f^k)=\sum_{i=1}^n \varepsilon_i \lambda_i^k,$$such that all $\lambda_i$'s or all $-\lambda_i$'s  are eigenvalues of $ \bigwedge \D_\ast$.



\begin{lemma}\label{lemvorm}
If $f$ is a non-nilpotent hyperbolic map on an infra-nilmanifold, with $(\dd,\D)$ as affine homotopy lift, it is possible to write$$N(f^k)=\sum_{i=1}^{m} a_i \lambda_i^k,$$with $a_i \in \Z$, $a_1\geq 1$ and such that $$|\lambda_1|=\lambda_1=\Sp(\bigwedge \D_\ast)>|\lambda_2|\geq \dots \geq |\lambda_m|.$$
\end{lemma}
\begin{proof}
By previous arguments, we know that it is possible to write $$N(f^k)=\sum_{i=1}^n \varepsilon_i \lambda_i^k,$$where all $\lambda_i$'s or all $-\lambda_i$'s  are eigenvalues of $ \bigwedge \D_\ast$. By grouping the $\lambda$'s that appear more than once and by changing the order, we get the desired form $$N(f^k)=\sum_{i=1}^{m} a_i \lambda_i^k,$$with $a_i \in \Z$ and $|\lambda_1|\geq |\lambda_2|\geq \dots \geq |\lambda_m|.$ There is a unique eigenvalue of $\bigwedge \D_\ast$ of maximal modulus, namely the product $$\prod_{|\lambda_i|\geq 1}\lambda_i=\mu_1.$$Note that it is clearly real, because for every $\lambda \not \in \R$, we know that if $|\lambda|>1$, then $|\overline{\lambda}|>1$ and both are eigenvalues of $\bigwedge \D_\ast$, because $\D_\ast$ is a real matrix. It is unique because $f$ is hyperbolic and $\D_\ast$ has no eigenvalues of modulus $1$.

Because of Theorem \ref{ThmAsymNielsen}, we know that $N^\infty(f)=\Sp(\bigwedge \D_\ast)=|\mu_1|$. Suppose now that $\mu_1$ or $-\mu_1$ does not appear as one of the $\lambda$'s in the expression of $N(f^k)$. Then, it should still hold that $$1=\limsup_{k\to \infty} \left(\frac{\sum_{i=1}^{m} a_i \lambda_i^k}{\mu_1^k}\right)^\frac{1}{k}.$$Let $a_{\max}=\max\{|a_i|\}$, then it is easy to derive that for all $k$:$$\frac{\sum_{i=1}^{m} a_i \lambda_i^k}{\mu_1^k}\leq \sum_{i=1}^{m} |a_i| \left|\frac{\lambda_i}{\mu_1}\right|^k\leq m a_{\max}  \left|\frac{\lambda_1}{\mu_1}\right|^k.$$So, we would have that$$1\leq\limsup_{k\to \infty} \left(m a_{\max}  \left|\frac{\lambda_1}{\mu_1}\right|^k\right)^\frac{1}{k}= \left|\frac{\lambda_1}{\mu_1}\right|<1,$$where the last inequality follows from the fact that $\mu_1$ is the unique eigenvalue of maximal modulus.  Moreover, an easy argument shows that $a_1 <0$ or $\lambda_1<0$ cannot occur in the expression of $N(f^k)$, because otherwise $N(f^k)$ would be negative for sufficiently large $k$. As we already proved that $a_1=0$ is impossible, we know that $a_1\geq 1$ and that $\Sp(\bigwedge \D_\ast)$ will appear as one of the $\lambda$'s in the expression for $N(f^k)$.
\end{proof}

\begin{remark}
The fact that $\Sp(\bigwedge \D_\ast)$ has to appear in the expression for $N(f^k)$ was proved in a more general setting in \cite{fl14-1}.
\end{remark}

\begin{lemma}\label{lemNfnot0}
When $f$ is a hyperbolic map on an infra--nilmanifold, then $N(f^k)\neq 0$, for all $k>0$.
\end{lemma}
\begin{proof}
Let $(\dd,\D)$ be an affine homotopy lift of $f$ and let $F$ be the holonomy group of the infra--nilmanifold. By \cite{ll06-1}, we know that $$N(f^k)=\frac{1}{\#F}\sum_{\A\in F}|\det(I-\A_\ast\D_\ast^k)|.$$Because all the terms have a non-negative contribution to this sum, we know$$N(f^k)\geq \frac{1}{\#F} |\det(I-\D_\ast^k)|=\frac{1}{\#F} \prod_{i=1}^{n}|1-\mu_i^k|>0,$$where the $\mu_i$ are all the eigenvalues of $\D_\ast$. The last inequality follows from the fact that $f$ is hyperbolic and so there are no eigenvalues of modulus $1$.
\end{proof}

From now on, we will consider $f$ to be a hyperbolic map on an infra-nilmanifold and $N(f^k)$ to be of the form $$N(f^k)=\sum_{i=1}^{m} a_i \lambda_i^k,$$with $a_i \in \Z$, $a_1\geq 1$ and such that $$|\lambda_1|=\lambda_1=\Sp(\bigwedge \D_\ast)>|\lambda_2|\geq \dots \geq |\lambda_m|.$$For the sake of clarity, we will keep using this notation in the rest of this paragraph.
\begin{lemma}\label{lemmu}
For all $\mu$ such that $\lambda_1>\mu >1$, there exists a $k_0\in \N$, such that for all $k\geq k_0$ and for all $n\in \N$, we have the following inequality:$$N(f^{k+n})>\mu^n N(f^k).$$
\end{lemma}
\begin{proof}
Pick $1>\varepsilon>0$, such that $$\frac{\lambda_1-\mu}{\lambda_1+\mu}\geq \varepsilon>0.$$Note that this implies that $\lambda_1\frac{1-\varepsilon}{1+\varepsilon}\geq \mu$. Now, choose $k_0\in \N$, such that for all $i\in \{2,\dots ,m\}$,$$\left|\frac{a_i}{a_1}\right|\left|\frac{\lambda_i}{\lambda_1}\right|^{k_0}< \frac{\varepsilon}{m}.$$Because of Lemma \ref{lemvorm}, we know that $|\lambda_1|>|\lambda_i|$, for all these $i$'s, so by choosing $k_0$ large enough, this will be possible. Moreover, the same inequalities will also hold for every $k\geq k_0$.
\medskip

Now, consider the fraction $$\frac{N(f^{k+n})}{N(f^k)}= \frac{a_1\lambda_1^{k+n}+\sum_{i=2}^{m} a_i \lambda_i^{k+n}}{a_1\lambda_1^{k}+\sum_{i=2}^{m} a_i \lambda_i^{k}}=\frac{\lambda_1^{n}+\sum_{i=2}^{m} \frac{a_i}{a_1} \frac{\lambda_i}{\lambda_1}^{k}\lambda_i^n}{1+\sum_{i=2}^{m} \frac{a_i}{a_1} \frac{\lambda_i}{\lambda_1}^{k}}.$$Note that $N(f^k)\neq 0$, according to Lemma \ref{lemNfnot0}, so the fraction above is well-defined. It is now easy to see that the equality above implies the following inequalities:$$\frac{N(f^{k+n})}{N(f^k)}\geq \frac{\lambda_1^{n}-\left|\sum_{i=2}^{m} \frac{a_i}{a_1} \frac{\lambda_i}{\lambda_1}^{k}\right|\lambda_1^n}{1+\left|\sum_{i=2}^{m} \frac{a_i}{a_1} \frac{\lambda_i}{\lambda_1}^{k}\right|}> \lambda_1^n\frac{1-\varepsilon}{1+\varepsilon}\geq \lambda_1^n\left(\frac{1-\varepsilon}{1+\varepsilon}\right)^n\geq \mu^n.$$
\end{proof}

\begin{corollary}\label{Cornu}
There exists a $\nu$, such that $\lambda_1>\nu >1$ and an $l_0\in \N$, such that for all $l\geq l_0$ and for all $k<l$:$$N(f^l)>\nu^{l-k}N(f^k).$$
\end{corollary}
\begin{proof}
Fix a $\mu$ as in Lemma \ref{lemmu} and let $k_0$ be the resulting integer from this lemma. Note that Lemma \ref{lemmu} actually tells us that the sequence $\left\{N(f^k)\right\}_{k=1}^\infty$ will be strictly increasing from a certain point onwards. Because all Nielsen numbers are integers, this means that there will exist an $l_0\geq k_0$, such that $N(f^{l_0})>N(f^{l})$, for all $l<l_0$, so also for all $l<k_0$.

\medskip

Now, let us define the following number $$\tau=\min\left\{\left(\frac{N(f^{l_0})}{N(f^l)}\right)^{\frac{1}{l_0-l}}\| l<l_0 \right\}.$$It is clear that $\tau>1$. Take $\nu= \min\left\{\mu, \frac{1+\tau}{2}\right\}$. 
Clearly, $\lambda_1>\nu>1$ and for all $k<l_0$, we have the following inequalities:$$\frac{N(f^{l_0})}{N(f^k)}\geq \tau^{l_0-k}>\nu^{l_0-k}.$$
Because of Lemma \ref{lemmu} and the fact that $\mu\geq \nu$, we know this inequality also applies for all $l\geq l_0$.
\end{proof}

\begin{theorem}
If $f$ is a hyperbolic map on an infra-nilmanifold, with affine homotopy lift $(\dd,\D)$, such that $\D_\ast$ is not nilpotent, 
then there exists an integer $m_0$, such that $$[m_0,+\infty)\subset \HPer(f).$$
\end{theorem}
\begin{proof}
Choose $\nu$ and $l_0$ as in Corollary \ref{Cornu}. Since $$\lim_{k\to \infty} \frac{\nu^{2^{k-1}}}{k}=+\infty,$$we know there exists a $k_0$, such that $\forall k\geq k_0$: $\nu^{2^{k-1}}>k$. Define $m_0=\max\{2^{k_0}, 2l_0+1\}$. 

\medskip

Now, suppose that $m\geq m_0$ and $m$ is even. Let $K$ denote the number of different prime divisors of $m$. As $m\geq 2l_0+1$, we know that $\frac{m}{2}>l_0$ and hence, the result of Corollary \ref{Cornu} applies. Therefore, we have the following inequalities$$\sum_{p \textrm{ prime}, p|m} N(f^{\frac{m}{p}})\leq K\cdot N(f^{\frac{m}{2}})<\frac{K}{\nu^{\frac{m}{2}}}\cdot N(f^m).$$By Theorem \ref{ThmABLSS}, it now suffices to show that$$\frac{K}{\nu^{\frac{m}{2}}}\leq 1.$$Because $K$ denotes the number of different prime divisors of $m$, we certainly know that $m>2^K$. By the definition of $m_0$, we also know that $m\geq 2^{k_0}$. If $K\geq k_0$, then$$\nu^{\frac{m}{2}}>\nu^{2^{K-1}}>K,$$which is sufficient. If $k_0>K$, we have that$$\nu^{\frac{m}{2}}\geq \nu^{2^{k_0-1}}>k_0>K.$$So, when $m\geq m_0$ is even, $m\in \HPer(f)$.

\medskip

When $m\geq m_0$ is odd, a similar argument holds. Let $K$ again be the number of different prime divisors of $m$ and note that $m\geq 2l_0+1$ implies that 
$\frac{m-1}{2}\geq l_0$. Again, by using Corollary \ref{Cornu}, we get the following inequalities:$$\sum_{p \textrm{ prime}, p|m} N(f^{\frac{m}{p}})\leq K\cdot N(f^{\frac{m-1}{2}})<\frac{K}{\nu^{\frac{m+1}{2}}}\cdot N(f^m).$$Again, $m>2^K$ and by definition $m\geq 2^{k_0}$. When $K\geq k_0$, $$\nu^{\frac{m+1}{2}}>\nu^{\frac{2^K+1}{2}}>\nu^{2^{K-1}}>K.$$When $k_0 >K$, the same reasoning gives us$$\nu^{\frac{m+1}{2}}\geq\nu^{\frac{2^{k_0}+1}{2}}>\nu^{2^{k_0-1}}>k_0>K.$$This concludes the proof of this theorem.
\end{proof}

\begin{remark} Having obtained Lemma \ref{lemvorm}, it is also possible to prove our main theorem in an alternative way, by following the approach of 
\cite[Section 6]{fl13-1}
\end{remark}

\begin{remark}
Note that our proof also applies for every essentially irreducible map $f$ (on any manifold) for which there exists a $ \mu >1$ and a $k_0\in \N$, such that for all $k\geq k_0$ and for all $n\in \N$, we have that$$N(f^{k+n})>\mu^n N(f^k).$$This condition is therefore sufficient for $\HPer(f)$ to be cofinite in $\N$.
\end{remark}

\subsection{The nilpotent case}
For the sake of completeness, in this section we will also treat the case where $\D_\ast$ is nilpotent.

\medskip

The following two theorems can be found in \cite{dp11-1}.

\begin{theorem}\label{thmRInf}
Let $\Gamma\subseteq \Aff(G)$ be an almost-Bieberbach group with holonomy group $F\subseteq \Aut(G)$. Let $M=\Gamma\backslash G$ be the associated infra-nilmanifold.  If 
 $f:M\to M$ is a map with affine homotopy lift $(\dd, \D)$, then
\[
  R(f)=\infty \iff \exists \A \in F \text{ such that } \det(I - \A_\ast \D_\ast)=0.
\]
\end{theorem}

\begin{theorem}\label{thmN=R}
Let $f$ be a map on an infra-nilmanifold, such that $R(f)<\infty$, then $$N(f)=R(f).$$
\end{theorem}

\begin{proposition}\label{propnilp}
When $f$ is a hyperbolic map on an infra-nilmanifold with affine homotopy lift $(\dd,\D)$, such that $\D_\ast$ is nilpotent, then for all $k$$$N(f^k)=R(f^k)=1.$$ 
\end{proposition}
\begin{proof}
By combining Theorem \ref{thmRInf} and Theorem \ref{thmN=R} we know that every fixed point class of $f^k$ is essential if and only if for all $\A\in F$ (where $F$ is the holonomy group of our infra-nilmanifold),$$\det(I-\A_\ast\D_\ast^k)\neq 0.$$By Lemma 3.1 in \cite{ddm05-1}, we know that there exists $\B \in F$ and an integer $l$, such that $$(\B_\ast\D_\ast^k)^l=\D_\ast^{lk} \textrm{ and } \det(I-\A_\ast\D_\ast^k)=\det(I-\B_\ast\D_\ast^k).$$Note that $\det(I-\B_\ast\D_\ast^k)=0$ implies that $\B_\ast\D_\ast^k$ has an eigenvalue $1$, but this would mean that $\D_\ast^{lk}$ has an eigenvalue $1$, which is in contradiction with the hyperbolicity of our map. Therefore, $R(f^k)=N(f^k)$.

\medskip

Note that $\D_\ast$ only has eigenvalues $0$. The fact that there exists $\B \in F$ and an integer $l$, such that $$(\B_\ast\D_\ast^k)^l=\D_\ast^{lk} \textrm{ and } \det(I-\A_\ast\D_\ast^k)=\det(I-\B_\ast\D_\ast^k),$$ implies that $\B_\ast\D_\ast^k$ only has eigenvalues $0$. As a consequence $\det(I-\A_\ast\D_\ast^k)=\det(I-\B_\ast\D_\ast^k)=1$, for all $\A\in F$. By applying the main formula from \cite{ll06-1}, an easy computation shows that $N(f^k)=1$.
\end{proof}
In \cite{fl13-1}, we can find the following proposition.

\begin{proposition}\label{propaff}
If $\overline{(\dd,\D)}:M\to M$ is a continuous map on an infra-nilmanifold, induced by an affine map, then every non-empty fixed point class is path-connected and
\begin{enumerate}
\item every essential fixed point class of $\overline{(\dd,\D)}$ consists of exactly one point.
\item every non-essential fixed point class of $\overline{(\dd,\D)}$ is empty or consists of infinitely many points.
\end{enumerate} 
\end{proposition}

\begin{theorem}
When $f$ is a hyperbolic map on an infra-nilmanifold with affine homotopy lift $(\dd,\D)$, such that $\D_\ast$ is nilpotent, then $$\HPer(f)=\{1\}.$$
\end{theorem}
\begin{proof}
Let $\overline{(\dd,\D)}$ be the induced map of $(\dd,\D)$ on our infra-nilmanifold. It suffices to show that $\Per(\overline{(\dd,\D)})=\{1\}$, because $N(f)=1$ immediately implies that $1\in \HPer(f)$.

\medskip

By Proposition \ref{propaff} and Proposition \ref{propnilp}, we know that $\Fix(\overline{(\dd,\D)}^k)$ consists of precisely one point, for all $k>0$. Because, for all $k>0$, it holds that $$\Fix(\overline{(\dd,\D)})\subset \Fix(\overline{(\dd,\D)}^k),$$we know that $\Fix(\overline{(\dd,\D)}^k)=\Fix(\overline{(\dd,\D)})$, for all $k>0$. From this it follows, that  $\overline{(\dd,\D)}$ only has periodic points of pure period $1$.
\end{proof}
\bibliography{G:/algebra/ref}

\begin{thebibliography}{10}

\bibitem{ablss95-1}
Alsed\`a, L., Baldwin, S., Llibre, J., Swanson, R., and Szlenk, W.
\newblock {\em Minimal sets of periods for torus maps via Nielsen numbers}.
\newblock Pacific J. Math., 1995, 169 1, 1--32.

\bibitem{brow82-1}
Brown, K.~S.
\newblock {\em Cohomology of groups.}, volume~87 of {\em Grad. Texts in Math.}
\newblock Springer-Verlag New York Inc., 1982.

\bibitem{brow71-1}
Brown, R.~F.
\newblock {\em The Lefschetz fixed point theorem}.
\newblock Scott, Foresman and Company, 1971.

\bibitem{ddm05-1}
Dekimpe, K., De~Rock, B., and Malfait, W.
\newblock {\em The {A}nosov relation for {N}ielsen numbers of maps of
  infra-nilmanifolds}.
\newblock Monatschefte f\"ur Mathematik, 2007, 150, pp. 1--10.

\bibitem{dd13-2}
Dekimpe, K. and Dugardein, G.-J.
\newblock {\em Nielsen zeta functions for maps on infra-nilmanifolds are
  rational}.
\newblock Journal of Fixed Point Theory and Applications, 2014.

\bibitem{dp11-1}
Dekimpe, K. and Penninckx, P.
\newblock {\em The finiteness of the {R}eidemeister number of morphisms between
  almost-crystallographic groups}.
\newblock J. Fixed Point Theory Appl., 2011, 9 2, 257--283.

\bibitem{fels00-2}
Fel'shtyn, A.
\newblock {\em Dynamical zeta functions, {N}ielsen theory and {R}eidemeister
  torsion}.
\newblock Mem. Amer. Math. Soc., 2000, 147 699, xii+146.

\bibitem{fl13-1}
Fel'shtyn, A. and Lee, J.~B.
\newblock {\em The Nielsen and the Reidemeister Zeta Functions of maps on
  infra-solvmanifolds of type (R)}.
\newblock arXiv:1303.0784.

\bibitem{fl14-1}
Fel'shtyn, A. and Lee, J.~B.
\newblock {\em The Nielsen numbers of iterations of maps on infra-solvmanifolds
  of type (R) and periodic points}.
\newblock arXiv:1403.7631.

\bibitem{fels88-1}
Fel'shtyn, A.~L.
\newblock New zeta functions for dynamical systems and {N}ielsen fixed point
  theory.
\newblock In {\em Topology and geometry---{R}ohlin {S}eminar}, volume 1346 of
  {\em Lecture Notes in Math.}, pages 33--55. Springer, Berlin, 1988.

\bibitem{hk97-1}
Heath, P. and Keppelmann, E.
\newblock {\em Fibre techniques in Nielsen periodic point theory on nil and
  solvmanifolds I}.
\newblock Topology and its Applications, 1997, 76 pp. 217--247.

\bibitem{jm02-1}
Jezierski, J. and Marzantowicz, W.
\newblock {\em Homotopy minimal periods for nilmanifold maps}.
\newblock Mathematische Zeitschrift, 2002, 239 2, pp.381--414.

\bibitem{jm06-1}
Jezierski, J. and Marzantowicz, W.
\newblock {\em Homotopy Methods in Topological Fixed and Periodic Point
  Theory}, volume~3 of {\em Topological Fixed Point Theory and Its
  Applications}.
\newblock Springer, 2006.

\bibitem{jian83-1}
Jiang, B.
\newblock {\em Nielsen Fixed Point Theory}, volume~14 of {\em Contemp. Math.}
\newblock American Mathematical Society, 1983.

\bibitem{kian89-1}
Kiang, T.-h.
\newblock {\em The Theory of Fixed Point Classes}.
\newblock Springer-Verlag, 1989.

\bibitem{ll06-1}
Lee, J.~B. and Lee, K.~B.
\newblock {\em Lefschetz numbers for continuous maps and periods for expanding
  maps on infra-nilmanifolds}.
\newblock J. Geom. Phys., 2006, 56 10, pp. 2011--2023.

\bibitem{lz07-1}
Lee, J.~B. and Zhao, X.
\newblock {\em Homotopy minimal periods for expanding maps on
  infra-nilmanifolds}.
\newblock J. Mathem. Soc. of Japan, 2007, 59 1, pp. 179--184.

\bibitem{lee95-2}
Lee, K.~B.
\newblock {\em Maps on infra-nilmanifolds}.
\newblock Pacific J. Math., 1995, 168, 1, pp. 157--166.

\bibitem{nomi54-1}
Nomizu, K.
\newblock {\em On the cohomology of compact homogeneous spaces of nilpotent Lie
  groups.}
\newblock Ann. of Math., 1954, 59, pp. 531--538.

\bibitem{fp85-1}
Pilyugina, V.~B. and Fel'shtyn, A.~L.
\newblock {\em The {N}ielsen zeta function}.
\newblock Funktsional. Anal. i Prilozhen., 1985, 19 4, 61--67, 96.

\bibitem{smal67-1}
Smale, S.
\newblock {\em Differentiable dynamical systems}.
\newblock Bull. Amer. Math. Soc., 1967, 73, pp. 747--817.

\end{thebibliography}
\bibliographystyle{G:/algebra/ref}

\end{document}